\documentclass[12pt, english, a4paper]{amsart}
\usepackage[margin=2.5cm]{geometry}
\usepackage{amssymb,amsmath,amsthm}
\usepackage{commath}
\usepackage{mathtools}
\usepackage{mathrsfs}
\usepackage{accents}
\usepackage[utf8]{inputenc}
\usepackage[T1]{fontenc}

\usepackage{caption}
\usepackage{subcaption}

\usepackage{xfrac}
\usepackage{enumerate}
\usepackage{graphics, float}
\usepackage{fourier}

\usepackage[colorlinks]{hyperref}
\hypersetup{allcolors = black}

\usepackage{url}
\usepackage[T1]{fontenc} 

\usepackage{tikz}
\usepackage{tkz-euclide}
\usepackage{svg}

\usepackage{multirow}
\usepackage{array}
\usepackage{bbm}
\usepackage{verbatim}
\usepackage{tikz}
\usepackage{float}
\usepackage{enumerate}
\usepackage{diagbox}
\usepackage{soul}
\usepackage{mathtools}

\newtheorem{theorem}{Theorem}

\newtheorem*{question}{Question}
\newtheorem{proposition}{Proposition}
\theoremstyle{definition}
\newtheorem*{definition}{Definition}

\newtheorem{lemma}{Lemma}

\newcommand{\inte}{\mathrm{int}\,}
\newcommand{\0}{\mathbf{0}}

\newcommand{\R}{\mathbb{R}}
\newcommand{\eps}{\varepsilon}
\newcommand{\K}{\mathcal{K}}
\newcommand{\dist}{\mathrm{dist}}
\newcommand{\cl}{\mathrm{cl}}

\title{Covering spiky annuli by planks}
\author[Ambrus, Huddell, Lai, Quirk, Williams]{Gergely Ambrus, Julian Huddell, Maggie Lai, Matthew Quirk, Elias Williams}

\thanks{
Research of G. A. was partially supported by the ERC Advanced Grant "GeoScape" no.  882971, by Hungarian National Research (NKFIH) grants no. 147145, 147544,  and 150151, which has been implemented with the support provided by the Ministry of Culture and Innovation of Hungary from the National Research, Development and Innovation Fund, financed under the ADVANCED-24 funding scheme.
Research was supported by project no. TKP2021-NVA-09, which has been implemented with the support provided by the
Ministry of Innovation and Technology of Hungary from the National
Research, Development and Innovation Fund, financed under the
TKP2021-NVA funding scheme. This work was supported by the Ministry of Innovation and Technology NRDI Office within the framework of the Artificial Intelligence National Laboratory (RRF-2.3.1-21-2022-00004). 
}



\begin{document}

\begin{abstract}
Answering Tarski's plank problem, Bang showed in 1951 that it is impossible to cover a convex body $K \subset \mathbb{R}^d$ with $d \geq 1$ by planks whose total width is less than the minimal width $w(K)$ of $K$. In 2003, A. Bezdek asked whether the same statement holds if one is required to cover only the annulus obtained from $K$ by removing a homothetic copy contained within. He showed that if $K$ is the unit square, then saving width in a plank covering is not possible, provided that the homothety factor is sufficiently small. White and Wisewell in 2006 characterized polygons that possess this property. 

We generalize the constructive part of their classification to spiky convex bodies: a body $K$ is spiky at a boundary point $x$ with supporting hyperplane $H$ and corresponding outer normal $u$, if both $K$ and its tangent cone at $x$ intersect $H$ only at $x$.
 We show  that if $K$ is a convex disc or a convex body in 3-space that is spiky in a minimal width direction, then for every $\eps \in (0,1)$ it is possible to cut a homothetic copy $\eps K$ from the interior of $K$ so that the remaining annulus can be covered by planks whose total width is strictly less than $w(K)$.
\end{abstract}

\maketitle

\section{Introduction}
In 1932, Tarski \cite{Tarski} conjectured that if a convex body $K \subset \R^d$  is covered by a finite number of planks, then the sum of their widths is not less than the minimal width of $K$. This is now known as Tarski's plank problem. 

Henceforth, a \textit{plank} is the closed region of $\mathbb{R}^d$ between two parallel hyperplanes, whose distance apart is the \textit{width} of the plank. The family $\mathcal{K}^d$ of {\em convex bodies} consists of all convex, compact sets in $\R^d$ with nonempty interior, that is,   $\inte K \neq \emptyset$. The \textit{ minimal width} $w(K)$ of $K \in \mathcal{K}^d$ is the smallest possible width of a single plank that covers $K$, and such a plank is called a \textit{ minimal width plank} of $K$. A chord between two points of $K$ is called a \textit{minimal width chord} if it is perpendicular to a minimal width plank and has length $w(K)$. The existence of such a chord is a well-known fact\footnote{Considering a minimal width plank $P$ bounded by hyperplanes $H_1$ and $H_2$, one readily sees that $H_1 \cap K$ and the orthogonal projection of $H_2 \cap K$ onto $H_1$ must intersect, as otherwise taking a $(d-2)$-dimensional flat $L$ which separates them, one could rotate and shrink $P$ about $L$  so as to reach a plank of smaller width containing $K$. }.

In 1951, Bang \cite{Bang} proved Tarski's conjecture and, at the same time, formulated a natural strengthening of it: the affine plank problem. To date, this latter conjecture has only been proved for symmetric convex bodies \cite{ball1991plank}. The crux of Bang's proof is, given a system of planks $\mathcal{P}$, the construction of a discrete point set, the so-called {\em Bang system}, whose elements cannot be simultaneously covered by the interiors of members of $\mathcal{P}$.

In 2003, A. Bezdek \cite{Bezdek} proposed the following variant of the plank problem:
\begin{question}[Plank problem for annuli]
    Let $K \in \mathcal{K}^d$ with $0 \in \inte K$. Is it true that for sufficiently small $\eps \in (0,1)$, the annulus $K \setminus \varepsilon K$ cannot be covered by a finite system of planks whose total width is less than $w(K)$?
\end{question}

Note that the criterion $0 \in \inte K$ guarantees that $\varepsilon K$ is a homothetic copy of $K$ that is contained in $\inte K$.

The most interesting case of the above question is obtained by setting $K$ to be the unit disc in the plane centered at the origin. Then $\eps K$ is a tiny concentric disc, and it is natural to expect that covering the resulting circular annulus still requires total width 2. Although the question for circular discs is still entirely open, there has been some progress for other choices of $K$. Our short note continues this line of results. 

A. Bezdek showed~\cite{Bezdek} that the annulus plank problem is answered in the affirmative for squares (with $\eps \leq 1 - 1/\sqrt{2}$) and for polygons whose inradius is exactly half of their minimal width. That is, he showed that these polygons, after removal of a small homothetic copy, still cannot be covered by planks whose total width is less than the minimal width of the polygon. His argument uses a clever alteration of the Bang point system. By essentially the same reasoning, in 2010 Smurov, Bogataya, and Bogatyi \cite{Smurov} extended Bezdek's theorem for squares to unit $d$-dimensional cubes. In 2006, White and Wisewell \cite{White-Wisewell} gave a complete characterization of planar polygons for which the annulus plank problem has an affirmative answer:
\begin{theorem}[White and Wisewell, 2006] \label{ww}
    Let $P$ be a convex polygon.
    \begin{itemize}
        \item If there is a minimal width chord of $P$ that meets a vertex of $P$ and divides the angle at that vertex into two acute angles, then for every $\varepsilon >0$ an $\varepsilon$-scaled copy of $P$ can be removed from $\inte P$ so that the resulting annulus can be covered by finitely many planks of total width strictly less than $w(P)$.
        \item If there is no such minimal width chord, then the removal of any set of sufficiently small diameter from $\inte P$ yields an annulus that cannot be covered by planks whose total width is less than $w(P)$.
    \end{itemize}
\end{theorem}

For the first part of the result, White and Wisewell showed that if a convex polygon has a minimal width direction in which it is ``pointy'', then an $\varepsilon$-scaled copy of $K$ can be placed very close to that vertex along the minimal width chord so that width can be saved in a plank covering of the resulting annulus. The spikiness property, that we formally define below, extends pointiness beyond polygons.

First, let us introduce some relevant terms.  We will work in $\R^d$ with $d \geq 2$. For a set $A$, we denote by $\inte A$, $\cl A$, $\partial A$, and $A^c$  its interior, closure, boundary, and complement, respectively.
Given a nonzero vector $v\in \R^d$ and some $t \in \R$, we introduce the hyperplane $H(v,t) = \{ x\in \mathbb{R}^d \mid \langle x, v \rangle = t\}$. For a convex body $K \in \mathcal{K}^d$, its support function is defined on the unit sphere $S^{d-1}$ by $h_K(u) = \max_{x\in K}  \langle x, u\rangle$. 
For $u \in S^{d-1}$, the supporting hyperplane of $K\in \mathcal{K}^d$ with outer normal $u$ is $H_K(u) = H(u, h_K(u))$. For a boundary point $x\in \partial K$,  the tangent cone of $K$ at $x$, denoted by $T_K(x)$, is \[T_K(x) = \cl \{ x + \alpha(y-x) \mid y\in K, \alpha\geq 0\}.\]

In the following, we provide a formal definition for the spiky property -- see Figure~\ref{fig:spikiness}.
\begin{definition}
    $K \in \mathcal{K}^d$ is \textit{spiky} in direction $u \in S^{d-1}$ if $K \cap H_K(u)$ is a single point $x \in \partial K$ at which $T_K(x) \cap H_K(u) = \{x\}$. 
\end{definition}


\begin{figure}[h]
\centering
\begin{subfigure}{.5\textwidth}
  \centering
  \includegraphics[width=0.75\linewidth]{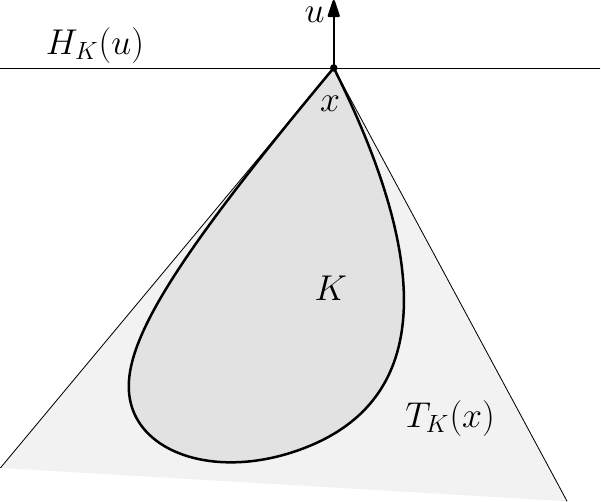}
  \caption{$K$ is spiky in direction $u$}
  \label{fig:spiky}
\end{subfigure}%
\begin{subfigure}{.5\textwidth}
    \centering
    \includegraphics[width=0.85\linewidth]{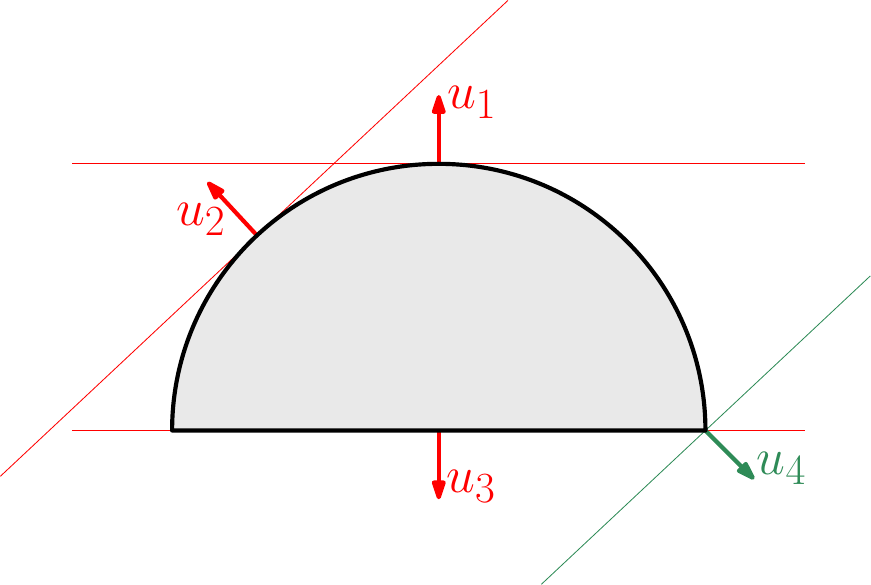}
    \caption{The semicircle is spiky in direction $u_4$, but not in directions $u_1, u_2, u_3$}
    \label{fig:semicircle}
\end{subfigure}
\caption{Spikiness}
\label{fig:spikiness}
\end{figure}

Our goal is to extend the first part of Theorem~\ref{ww} to spiky convex bodies in both $\mathbb{R}^2$ and $\mathbb{R}^3$ by showing that a well-placed homothetic copy can be cut from these in order to economize a plank covering. We refer to such an annulus as spiky.

\begin{theorem}\label{spikythm}
    If $K \in \mathcal{K}^d$, $d=2,3$, has a minimal width direction in which it is spiky, then for all $\varepsilon \in (0, 1)$ there exists $y\in \mathbb{R}^d$ with $\varepsilon K + y\subset \inte K$ such that $K\setminus \{\varepsilon K + y\}$ can be covered by planks of total width strictly less than $w(K)$.
\end{theorem}

In particular, the result applies to the Reuleaux triangle and the Meissner body with properly selected minimal width directions, but it does not hold for smooth convex bodies. 

We prove the theorem in Section~\ref{sec:proof} using a quantitative covering result on planar annuli given in Section~\ref{sec:planar}.

\section{Technical lemmas}  \label{sec:planar}

Given a spiky convex body, we assume without loss of generality that it is in \textit{standard position} meaning that $K$ is spiky in direction $-e_d$, the minimal width of $K$ is $1$, and
\[K \cap H_K(-e_d) = \{\0 \},\] see Figure~\ref{fig:standardpos}. In addition, we will write $T_K \coloneqq T_K(\0)$. Let $ H_t \coloneqq H(e_d,t)$ be the horizontal hyperplane at height $t$. Note that $H_0 = H_K(-e_d)$ and $H_1=H_K(e_d)$ are supporting hyperplanes of~$K$.

\begin{figure}[h]
    \centering
    \includegraphics[width=0.65\linewidth]{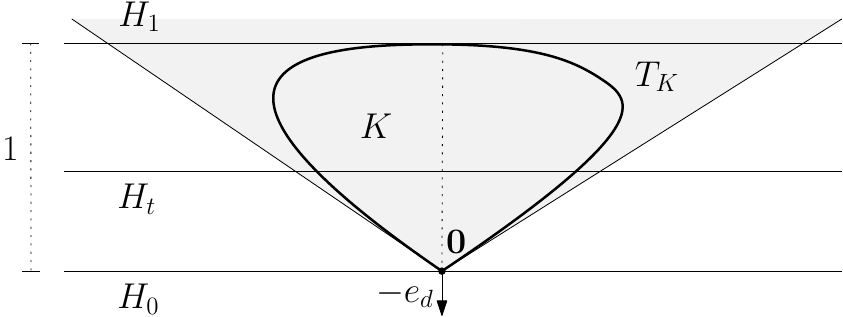}
    \caption{The spiky convex body $K$ in standard position}
    \label{fig:standardpos}
\end{figure}

Now, let $x \in \R^d$ and $A \subset \R^d$. The distance of $x$ from $A$ is, as usual, 
\[
\dist(x, A) = \inf_{y \in A} |x -y|,
\]
where $|.|$ is the Euclidean norm. For $A,B$ compact nonempty sets in $\K^d$, their \textit{Hausdorff distance} $d_H(A,B)$ is defined as
\[d_H(A,B)=\max\left(\sup_{a\in A} \dist(a,B),\sup_{b\in B}\dist(b,A)\right).\]
Given some $K \in \K^d$, we introduce its {\em metric annulus} of width $\varepsilon$ by
\[
K^\varepsilon=\{x\in K\mid \dist(x,K^c)\le\varepsilon\}
\]
which is the complement of the open inner parallel body within $K$.
Note that if $K,L\in \mathcal{K}^d$ with $L \subset K$, then $K\setminus L\subset K^\varepsilon$ if and only if $d_H(\partial L, \partial K)= d_H(K\setminus L, \partial K)\le \varepsilon$.
\begin{lemma}\label{lem-1}
    Let $K\in \mathcal{K}^d$ be a spiky convex body in standard position. As $t\to 0$ along positive values, the set $\frac{1}{t} ((T_K\setminus \inte K )\cap H_t)$ converges to $ \partial T_K \cap H_1$ in the Hausdorff distance.    
\end{lemma}

The second tool, which forms the crux of the proof, is an economic cap covering theorem for planar convex discs in the sense that the metric annulus is to be covered by few caps (for related results, see \cite{barany}).

\begin{lemma}\label{lem-2}
    Let $K \in \mathcal{K}^2$ with perimeter $\rho$ and let $0< \delta < w(K)$. Then $K^\delta$ can be covered by at most $\sqrt{\frac{2\pi \rho}{\delta}}$ planks of width $2\delta$.
\end{lemma}

\begin{proof}[Proof of Lemma \ref{lem-1}]
Note that for each $t >0$,  $\frac{1}{t} ((T_K\setminus\inte K)\cap H_t) \subset H_1$. By convexity, $\frac{1}{t} (\inte K \cap H_t)$ increases and $\frac{1}{t} ((T_K\setminus\inte K )\cap H_t)$ decreases as $t\searrow 0$, and by spikiness, each such set is bounded and therefore compact. Since a nested sequence of nonempty compact sets in $\R^d$ converges to their intersection in the Hausdorff metric (see e.g. \cite{petrunin2023space}), by monotonicity we derive that the limit exists in $H_1$:
    \[
    T:= \lim_{t \to 0} \frac 1 t \left(\left(T_K\setminus\inte K \right)\cap H_{t}\right) = \bigcap_{t > 0 } \frac 1 t  \left(\left(T_K\setminus\inte K\right)\cap H_{1/t}\right).
    \]
We clearly have $\partial T_K\cap H_1 \subset T$. Suppose for the sake of contradiction that $T \neq \partial T_K\cap H_1$. Then there exists some $x\in \inte T_K\cap H_1$ such that for all $t>0$, $tx\notin \inte K$. Thus, for each $t \in (0,1)$ one may separate $x$ and $ \frac 1 t (\inte K \cap H_t)$  in $H_1$ by a $(d-2)$-dimensional affine subspace. Since the latter are increasing as $t\searrow 0$, a standard compactness argument shows the existence of a supporting $(d-2)$-flat of $T_K \cap H_1$ through $x$, a contradiction. 
\end{proof}

\begin{proof}[Proof of Lemma \ref{lem-2}]
We first claim that it is possible to cover $\partial K$ with at most  $\sqrt{\frac{2\pi \rho}{\delta}}$ planks of width at most $\delta$. For proving this, we recursively define a set of planks $(P_i)_1^n$ with the desired properties, see Figure~\ref{fig:planks}.

To start with, let $p_1 \in \partial K$ be arbitrary. Next, for every $i \geq 1$,  take  a supporting line $\ell_{i}$ of $K$ at $p_{i}$, consider the line $\ell_{i}'$ obtained by translating $\ell_{i}$ by distance $\delta$ towards $K$, and define $P_i$ to be the plank of width $\delta$ bounded by $\ell_i$ and $\ell_i'$. Moreover, let $p_{i+1}$ be the intersection point between $\ell_{i}'$ and $\partial K$ in the counterclockwise direction from $p_{i}$. Denote by $\widearc{p_i p_{i+1}}$ the resulting arc of $ \partial K$ between these points, and by $|\widearc{p_i p_{i+1}}|$ its arclength.

We stop the process when the planks cover $\partial K$: let $n$ be the smallest index for which $\partial K \subset \bigcup_{i=1}^n P_i$. Such an $n$ always exists since  $|\widearc{p_i p_{i+1}}| \geq \delta$ for every index $i\geq 1$. 

If $p_1 \not \in P_n$, we consider the planks $P_1, \ldots, P_n$. Otherwise, we re-define $p_{n+1}:= p_1$, $\ell_{n+1}:= \ell_1$,  and $P_n$ as the plank bounded  by $\ell_n$ and the parallel line through $p_1$. Note that in both cases, $p_n \not \in P_1$ by the minimality of $n$.

For each $i = 1, \ldots, n$, let  $\alpha_{i}$ be the angle between $\ell_{i}$ and $\ell_{i+1}$. Since $\ell_{i+1}$ supports $K$, and the arc $\widearc{p_i p_{i+1}}$ connects points on opposite sides of $P_i$, we derive that 
\begin{equation} \label{eq:arclengthestimate}
|\widearc{p_i p_{i+1}} \geq \frac{\delta}{\sin{\alpha_i}}
\end{equation}
for every $i = 1, \ldots, n$ in the first case and for every $i = 1, \ldots, n-1$ in the second case (as then $P_n$ has width less than $\delta$). Yet, $p_n \not \in P_1$ implies that \eqref{eq:arclengthestimate} holds for $i = n$ in this case as well.

\begin{figure}[h]
\centering
\begin{subfigure}{.5\textwidth}
  \centering
  \includegraphics[width=0.95\linewidth]{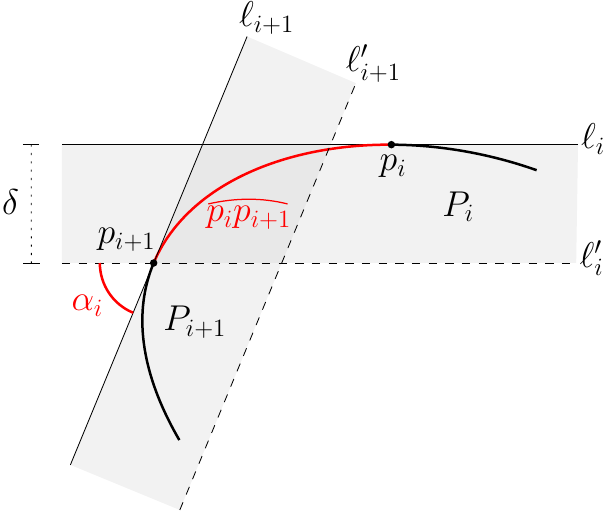}
  \caption{Two adjacent planks in a covering of $\partial K$}
  \label{good}
\end{subfigure}%
\begin{subfigure}{.5\textwidth}
  \centering
  \includegraphics[width=0.95\linewidth]{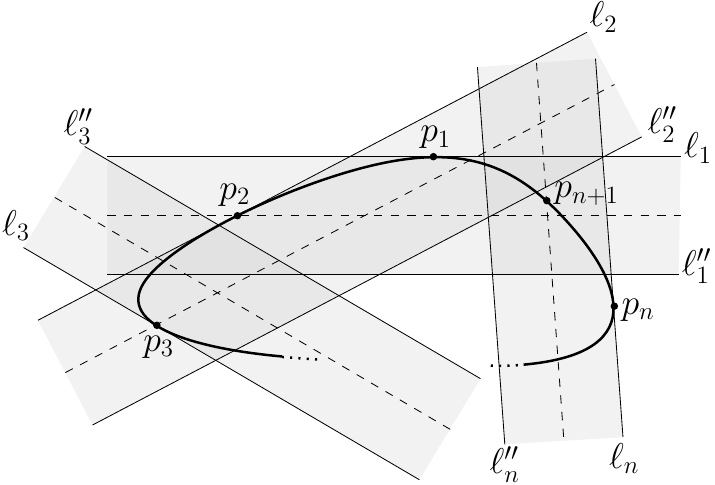}
  \caption{A covering requiring $n$ planks}
  \label{fig:bad}
\end{subfigure}
\caption{Construction of the system of planks covering $\partial K$}
\label{fig:planks}
\end{figure}

Since by construction, the arcs $\widearc{p_i p_{i+1}}$ for $i = 1, \ldots, n$ are pairwise nonoverlapping parts of $\partial K$, we derive from \eqref{eq:arclengthestimate} that 

    \begin{equation}\label{perimeter}
        \sum_{i=1}^n \frac{\delta}{\sin{\alpha_i}}\leq \rho.
    \end{equation}
Furthermore, 
    \begin{equation}\label{angles}
        \sum_{i=1}^n \alpha_i \leq 2\pi.
    \end{equation}
Since the function $\frac{\delta}{\sin{x}}$ is convex over $(0, \pi)$ and each $\alpha_i \in (0,\pi)$, Jensen's inequality~\cite{Jensen} yields that
    \begin{equation}\label{ineq_sin}
    \frac 1 n \cdot \sum_{i=1}^n \frac{\delta}{\sin(\alpha_i)} \geq  \frac{\delta}{\sin\left({\frac{\sum_{i=1}^n \alpha_i}{n}}\right)}.
    \end{equation}
Note that as $\rho \geq 2 w(K)$ and $\delta < w(K)$, we have $\sqrt{\frac{2\pi \rho}{\delta}} >3$. Thus, if $n \leq 3$, the claim is proved. Otherwise, since $n \geq 4$ and $\sin x$ is increasing on $[0, \frac \pi 2]$, \eqref{angles} shows that 
\[
\sin\left({\frac{\sum_{i=1}^n \alpha_i}{n}}\right) \leq  \sin \frac {2 \pi}{n}.
\]
    
Thus, \eqref{perimeter} and \eqref{ineq_sin} imply that
    $$\frac{\delta}{\sin\frac{2\pi}{n}}\leq \frac{\rho}{n}$$
Recall that $\sin{x} < x$ for all $x > 0$. Then a rearrangement gives $n < \sqrt{\frac{2\pi \rho}{\delta}}$.

To finish the proof, for each $i\in [n]$, consider the line $\ell_i''$ obtained by translating $\ell_i$ at distance $2\delta$  towards $K$. Define  $P_i'$ to be the plank bounded by $\ell_i$ and $\ell_i''$. The construction guarantees that $(P_i')_1^n$ covers the metric annulus $K^\delta$.
\end{proof}

\section{The covering  construction} \label{sec:proof}
\begin{proof}[Proof of Theorem \ref{spikythm}]
Once again, we assume that $K$ is in standard position. 

It suffices to show that for any $\eps \in (0,1)$, the set $K\setminus  \inte \varepsilon K$ can be covered by a finite system $\mathcal{S}$ of planks of total width strictly less than $w(K)$. 
Indeed, by continuity, this yields that for sufficiently small $\kappa >0$, the annulus $K \setminus \{\varepsilon K + \kappa e_d\}$ can be covered by a system of planks whose total width is still strictly less than $1$:  assume that $\mathcal{S}$ consists of $N$ planks, and inflate each member of $\mathcal{S}$ about its central line/plane so that the width increases by $2\kappa$. The resulting inflated planks cover $K \setminus \{\varepsilon K + \kappa e_d\}$ (note that $\mathcal{S}$ covers $\partial (\eps K)$) while the total width increases by $2 \kappa N$. Setting $\kappa$ small enough guarantees that the total width remains less than $w(K)$. 

In order to construct a covering of $K\setminus  \inte \varepsilon K$, let $t \in (0,1)$ be a parameter whose value we will fix later. The covering starts with a single plank of width $1-t$ bounded by $H_1$ and $H_t$. This covers the part of the annulus above $H_t$. Thus, it suffices to show that the part of $K\setminus  \inte \varepsilon K$ between $H_t$ and $H_0$ can be covered with a finite set of planks whose total width is strictly less than $t$. 

To this end, we will first show that when $d=2,3$, then for a suitable value of $t$, the $(d-1)$-dimensional annulus  $(T_K \setminus \inte \varepsilon K)\cap H_t$ can be covered by a finite set $\mathcal{P}$ of $(d-1)$-dimensional planks of total width less than~$t$. 

By convexity of $\varepsilon K$, the Hausdorff distance $\delta_t := d_H( (\partial T_K\cap H_t),(T_K\setminus \inte \varepsilon K)\cap H_t)$ is monotonically decreasing as $t\to 0$. Note that $T_K = T_{\inte \eps K}$, hence Lemma~\ref{lem-1} states that $\frac{\delta_t}{t}\to0$ as $t\to 0$. When $d=2$, choose $t$ so that $ \delta_t < t/2$. If $d = 3$, then let $t$ such that $\delta_t < \frac{t}{8\pi \rho}$ where $\rho$ is the perimeter of $T_K \cap H_1$. 

\begin{figure}[h]
    \centering
    \includegraphics[width=0.6\linewidth]{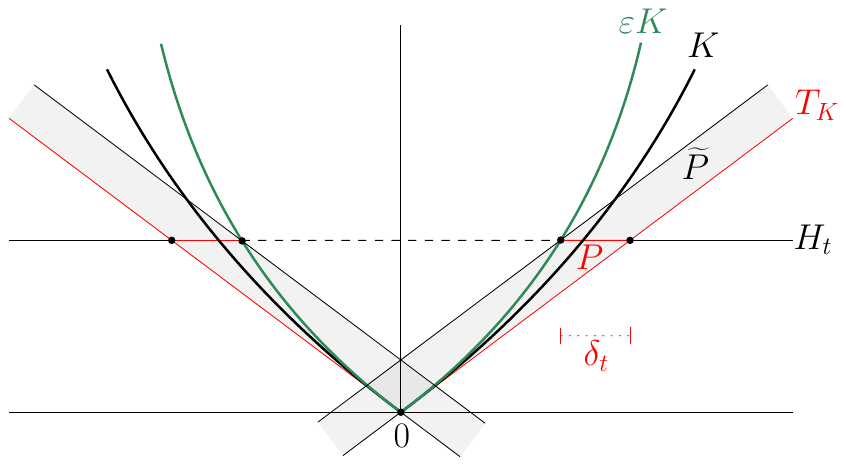}
    \caption{A covering strategy of a $2$-dimensional $K$}
    \label{fig:covering2D}
\end{figure}
\begin{figure}[h]
    \centering
    \includegraphics[width=0.6\linewidth]{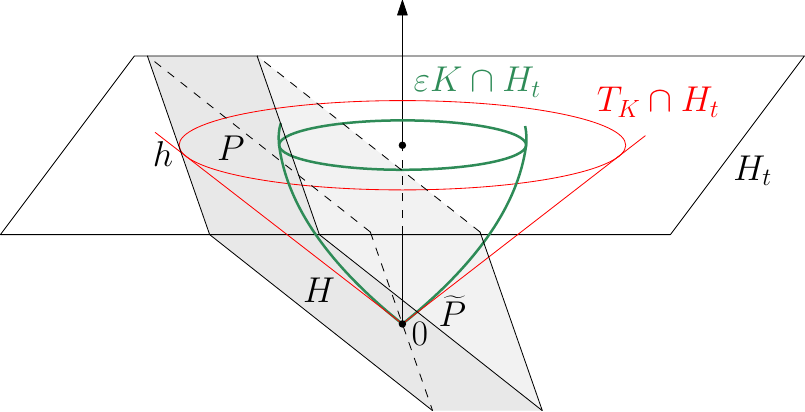}
    \caption{The covering construction $K\subset \R^3$}
    \label{fig:covering3D}
\end{figure}

By the definition of $\delta_t$,  covering the metric annulus $(T_K\cap H_t)^{\delta_t}$ also covers $(T_K \setminus \inte \varepsilon K)\cap H_t$. When $d=2$, then  $(T_K\cap H_t)^{\delta_t}$ can clearly be covered by two planks of width $\delta_t$, resulting in total width $2 \delta_1 < t$, hence this is a suitable choice for $\mathcal{P}$. For $d = 3$, note that the perimeter of $T_K \cap H_t$ is $t\rho$.  Hence, by Lemma \ref{lem-2}, the metric annulus $(T_K\cap H_t)^{\delta_t}$  can be covered by a system of planks whose total width is $\sqrt{8\pi (t\rho) \delta_t} < t$. Thus, setting $\mathcal{P}$ as this set works. 

Finally, we construct a set of $d$-dimensional planks from $\mathcal{P}$ that cover the lower part of $T_K\setminus  \inte \varepsilon K$. For each $P \in \mathcal{P}$, let $h$ be its bounding $(d-2)$-dimensional flat which is tangent to $T_K$, define $H$ to be the hyperplane (i.e. line/plane) through $h$ and $\0$, and consider the $d$-dimensional plank $\widetilde{P}$ which is bounded by $H$ such that $\widetilde{P} \cap H_t = P$ (see Figure~\ref{fig:covering2D} and Figure~\ref{fig:covering3D}). Note that $H$ is tangent to $T_K$ and that the width of $\widetilde{P}$ is at most that of $P$. Moreover, convexity of $\varepsilon K$ implies that the system $\widetilde{\mathcal P}$  of planks obtained from members of $\mathcal{P}$ covers $T_K\setminus  \inte \varepsilon K$, and thus $K\setminus  \inte \varepsilon K$ as well. As the total width of   $\widetilde{\mathcal P}$ is less than $t$, the proof is complete. 
\end{proof}

\section{Further Remarks}
The generalization of the second case of Theorem~\ref{ww} remains to be investigated, namely, whether the absence of a minimal width direction in which $K$ is spiky means that the removal of a sufficiently small homothetic copy of $K$ gives an annulus that still needs planks of total width at least the minimal width of $K$ to cover it.

Our current construction fails for general spiky convex bodies in $\mathbb{R}^d$. Although Lemma \ref{lem-1} is applicable for all $d\in \mathbb{N}$, Lemma \ref{lem-2} fails for $d \geq 3$, as in these cases, covering the unit sphere requires as much total width as covering the unit ball. Therefore, the construction used to prove Theorem~\ref{spikythm} does not generate an economic plank covering. 

However, our result does extend to higher dimensions in the special case where $T_K$ is polyhedral, i.e. it is the intersection of finitely many half-spaces that contain $\mathbf{0}$ on their boundaries.

\begin{proposition}
If $K \in \mathcal{K}^d$, $d\ge 2$, has a minimal width direction in which it is spiky and $T_K$ is polyhedral, then for all $\varepsilon \in (0, 1)$ there exists $y\in \mathbb{R}^d$ with $\varepsilon K + y\subset \inte K$ such that $K\setminus \{\varepsilon K + y\}$ can be covered by planks of total width strictly less than $w(K)$.
\end{proposition}
The proof is a simple modification of the proof of Theorem \ref{spikythm}.
\begin{proof}
The construction is identical to the one utilized in Theorem \ref{spikythm} with the only difference that $\mathcal{P}$ is now a set of planks that are parallel to the facets of $T_K\cap H_t$. Assuming that $T_K\cap H_t$ has $N$ facets, the defining inequality for $t$ is $\delta_t < \frac{t}{N}$. Then the metric annulus $(T_K\cap H_t)^{\delta_t}$  can be covered by $N$ planks of width $\delta_t$, which have total width less than $t$. The rest of the proof is unchanged. 
\end{proof}

\section{Acknowledgements}
We are grateful to the anonymous referee for suggestions that helped to improve the presentation. 
This research was done under the auspices of the Budapest Semesters in Mathematics program. 

\bibliographystyle{abbrv}
\bibliography{Plank_reference}

\bigskip
\noindent
{\sc Gergely Ambrus}
\smallskip

\noindent
{\em Bolyai Institute, University of Szeged, Hungary, \\ and HUN-REN Alfréd Rényi Institute of Mathematics,  Budapest, Hungary}

\smallskip
\noindent
 ORCID: 0000-0003-1246-6601
 
 \smallskip
\noindent
e-mail address: \texttt{ambrus@server.math.u-szeged.hu, ambrus@renyi.hu}

\medskip

\noindent
{\sc Julian Huddell}
\smallskip

\noindent
{\em Tulane University, New Orleans, Louisiana, United States}

\noindent
e-mail address: \texttt{jhuddell@tulane.edu}

\medskip
\noindent
{\sc Maggie Lai}
\smallskip

\noindent
{\em Tulane University, New Orleans, Louisiana, United States}

\noindent
e-mail address: \texttt{mlai2@tulane.edu}

\medskip

\noindent
{\sc Matthew Quirk}
\smallskip

\noindent
{\em Kalamazoo College, Kalamazoo, Michigan, United States}

\noindent
e-mail address: \texttt{Matthew.Quirk22@kzoo.edu}
\medskip

\noindent
{\sc Elias Williams}
\smallskip

\noindent
{\em Lewis \& Clark College, Portland, Oregon, United States }

\noindent
e-mail address: \texttt{roark@lclark.edu}

\end{document}